\documentclass[12pt, oneside, psamsfonts]{amsart}

\newif\ifPDF
\ifx\pdfoutput\undefined\PDFfalse
\else \ifnum \pdfoutput > 0 \PDFtrue
        \else \PDFfalse
        \fi
\fi

\usepackage[centertags]{amsmath}
\usepackage{amsfonts}
\usepackage{mathrsfs}
\usepackage{textcomp}
\usepackage{amssymb}
\usepackage{amsthm}
\usepackage{newlfont}
\usepackage[all]{xy}

\ifPDF
  \usepackage[pdftex]{color, graphicx}
  \usepackage[pdftex, bookmarks, colorlinks]{hyperref}
  \hypersetup{colorlinks=false}

\else
  \usepackage{color}
  \usepackage[dvips]{graphicx}
  \usepackage[dvips]{hyperref}
\fi

\usepackage[scale=0.8]{geometry}

\usepackage[pagewise, mathlines, displaymath]{lineno}

\newtheorem{thm}{Theorem}[section]
\newtheorem{cor}[thm]{Corollary}
\newtheorem{lem}[thm]{Lemma}

\theoremstyle{definition}
\newtheorem{defn}[thm]{Definition}
\theoremstyle{remark}
\newtheorem{rem}[thm]{Remark}

\numberwithin{equation}{section}

\newcommand{\norm}[1]{\left\Vert#1\right\Vert}
\newcommand{\abs}[1]{\left\vert#1\right\vert}

\newcommand{\Real}{\mathbb R}
\newcommand{\Int}{\mathbb Z}

\newcommand{\Ratn}{\mathbb Q}
\newcommand{\eps}{\varepsilon}

\newcommand{\Kzero}{\mathrm{K}_0}

\newcommand{\tr}{\mathrm{T}}
\newcommand{\MC}[2]{\mathrm{M}_{#1}(\mathrm{C}(#2))}

\newcommand{\aff}{\mathrm{Aff}}

\begin{document}


\title[Classification of C*-algebras]{On the classification of simple amenable C*-algebras with finite decomposition rank}

\author{George A.~Elliott}
\address{Department of Mathematics, University of Toronto, Toronto, Ontario, Canada~\ M5S 2E4}
\email{elliott@math.toronto.edu}

\author{Zhuang Niu}
\address{Department of Mathematics, University of Wyoming, Laramie, WY 82071, USA}
\email{zniu@uwyo.edu}

\thanks{The research of the first author was supported by a Natural Sciences
and Engineering Research Council of Canada (NSERC) Discovery Grant, and the research of the second author was supported by a Simons Foundation Collaboration Grant}
\dedicatory{Dedicated to Richard V.~Kadison on the occasion of his ninetieth birthday}


\begin{abstract}
Let $A$ be a unital simple separable C*-algebra satisfying the UCT. Assume that $\mathrm{dr}(A)<+\infty$, $A$ is Jiang-Su stable, and $\Kzero(A)\otimes \Ratn\cong \Ratn$. Then $A$ is an ASH algebra (indeed, $A$ is a rationally AH algebra).
\end{abstract}

\maketitle

\setcounter{tocdepth}{1}

\section{Introduction}
Let $A$ be a simple separable nuclear unital C*-algebra. In \cite{Matui-Sato-DR}, Matui and Sato showed that $A\otimes \mathrm{UHF}$ can be tracially approximated by finite dimensional C*-algebras (i.e., is TAF) if $A$ is quasidiagonal with unique trace. 

In this note, this result is enlarged upon as follows: the condition on the trace simplex is removed, at the cost of assuming the UCT, (still) finite nuclear dimension, and (still) that all traces are quasidiagonal---e.g., by assuming finite decomposition rank---see \cite{BBSTWW}---and (so far) of restricting the $\Kzero$-group to have torsion-free rank equal to one.
\begin{thm}\label{main-thm}
Let $A$ be a simple unital separable C*-algebra satisfying the UCT. If $A\otimes Q$ has finite decomposition rank and $\Kzero(A)\otimes \Ratn\cong\Ratn$, then $A\otimes Q \in\mathrm{TA}\mathcal I$ (see Definition \ref{defn-TAI}). In particular, $A\otimes \mathcal Z$ is classifiable, where $\mathcal Z$ is the Jiang-Su algebra (\cite{JS-Z}).
\end{thm}

This theorem can also be regarded as an abstract version (still in a special case) of the classification result of  \cite{GLN-TAS} and \cite{EGLN-ASH}, where any simple unital locally approximately subhomogeneous C*-algebra is shown to be rationally tracially approximated by Elliott-Thomsen algebras (1-dimensional noncommutative CW complexes) (\cite{EGLN-ASH}) and hence to be classifiable (\cite{GLN-TAS}).

\section{The main result and the proof}
In this note let us use $Q$ to denote the UHF algebra with $\Kzero(Q)\cong\Ratn$, and let us use $\mathrm{tr}$ to denote the canonical tracial state of $Q$.

\begin{defn}[N.~Brown, \cite{Bro-QDTr}]
Let $A$ be a unital C*-algebra, and denote by $\tr_{\mathrm{qd}}(A)$ the tracial states with the following property: For any $(\mathcal F, \eps)$, there is a unital completely positive map $\phi: A\to Q$ such that 
\begin{enumerate}
\item $\abs{\tau(a)-\mathrm{tr}(\phi(a))}<\eps$, $a\in\mathcal F$, and
\item $\norm{\phi(ab)-\phi(a)\phi(b)}<\eps$, $a, b\in\mathcal F$.
\end{enumerate}
\end{defn}
\begin{rem}
In the original definition of a quasidiagonal trace (Definition 3.3.1 of \cite{Bro-QDTr}), the UHF algebra $Q$ was replaced by a matrix algebra. It is easy to see that these two approaches are equivalent.
\end{rem}

Recall the tracial approximate uniqueness result of  \cite{DL-classification} and \cite{Lin-s-uniq}.
\begin{thm}[Theorem 4.15 of \cite{DL-classification}; Theorem 5.3 of \cite{Lin-s-uniq}]\label{stable-uniq}
Let $A$ be a simple, unital, exact, separable C*-algebra satisfying the UCT. For any finite subset $\mathcal F\subseteq A$ and any $\eps > 0$, there exist $n\in\mathbb N$ and a $\underline{\mathrm{K}}$-triple $(\mathcal P, \mathcal G, \delta)$ with the following property: For any admissible codomain $B$, and any three completely positive contractions $\phi, \psi, \xi: A\to B$ which are $\delta$-multiplicative on $\mathcal G$, with $\xi$ unital, $\phi$ and $\psi$ nuclear, and $\phi_\#(p)=\psi_\#(p)$ in $\underline{\mathrm{K}}(B)$ for all $p\in\mathcal P$, and such that $\phi(1)$ and $\psi(1)$ are unitarily equivalent projections, there exists a unitary $u\in\mathcal U_{n+1}(B)$ such that
$$
\norm{u^*\left(\begin{array}{cc} \phi(a) & \\ & n\cdot \xi(a) \end{array} \right)u-\left( \begin{array}{cc} \psi(a) & \\ & n \cdot \xi(a) \end{array} \right)} < \eps,\quad a\in\mathcal F.
$$
One may arrange that $u^*(\phi(1)\oplus n\cdot 1)u=\psi(1)\oplus n\cdot 1$. 
\end{thm}

\begin{rem}
In the theorem above, $n\cdot \xi(a)$ (or $n\cdot 1$) denotes the direct sum of $n$ copies of $\xi(a)$ (or $1$). This notation is also used in the proof of Corollary \ref{stable-uniq-Q} below. 
\end{rem}

\begin{rem}
In the theorem above (and also Corollary \ref{stable-uniq-Q} below), one assumes by convention that the finite subset $\mathcal G$ is sufficiently large and $\delta$ is sufficiently small that $[\phi(p)]$ is well defined for any $p\in\mathcal P$ if a map $\phi$ is $\delta$-multiplicative on $\mathcal G$.
\end{rem}


When $B=Q$, in fact one does not have to consider all the K-theory with coefficients. More precisely, one has
\begin{cor}\label{stable-uniq-Q}
Let $A$ be a simple, unital, exact, separable C*-algebra satisfying the UCT. For any finite subset $\mathcal F\subseteq A$ and any $\eps > 0$, there exist $n\in\mathbb N$ and a $\underline{\mathrm{K}}$-triple $(\mathcal P, \mathcal G, \delta)$, with $\mathcal P\subseteq \mathrm{Proj}_\infty(A)$, with the following property: For any three completely positive contractions $\phi, \psi, \xi: A\to Q$ which are $\delta$-multiplicative on $\mathcal G$, with $\phi(1)=\psi(1)=1_Q-\xi(1)$ a projection, $[\phi(p)]_0=[\psi(p)]_0$ in $\Kzero(Q)$ for all $p\in\mathcal P$, and $\mathrm{tr}(\phi(1))=\mathrm{tr}(\psi(1))<1/n$, where $\mathrm{tr}$ is the canonical tracial state of $Q$, there exists a unitary $u\in Q$ such that
$$
\norm{u^* (\phi(a)\oplus\xi(a))u-\psi(a)\oplus\xi(a)}<\eps,\quad a\in\mathcal F.
$$
\end{cor}

\begin{proof}
Applying Theorem \ref{stable-uniq} to $\mathcal F$ and $\eps>0$, one obtains $n_0\in\mathbb N$ and a $\underline{\mathrm{K}}$-triple $(\widetilde{\mathcal P}, \mathcal G, \delta)$ with the property of Theorem \ref{stable-uniq}. Set $$n_0+1=n\quad\textrm{and}\quad \widetilde{\mathcal P}\cap \mathrm{Proj}_\infty(A)=\mathcal P.$$ Let us show that $n$ and $({\mathcal P}, \mathcal G, \delta)$ have the desired property.

Let $\phi, \psi, \xi: A\to Q$ be completely positive contractions which are $\delta$-multiplicative on $\mathcal G$, with $\phi(1)=\psi(1)=1_Q-\xi(1)$ a projection, $[\phi(p)]_0=[\psi(p)]_0$ in $\Kzero(Q)$ for all $p\in\mathcal P$, and $\mathrm{tr}(\phi(1))=\mathrm{tr}(\psi(1))<1/n$.

Decompose $\xi$ approximately on $\mathcal F\subseteq A$ as a repeated direct sum $$\underbrace{\xi'\oplus\cdots\oplus\xi'}_{n_0},$$ where $\xi': A\to Q$ is again a completely positive contraction which is (necessarily, if the approximation is sufficiently good) $\delta$-multiplicative on $\mathcal G$, and $(\xi'\oplus\cdots\oplus\xi')(1_A)=\xi(1_A)$. Since $\mathrm{tr}(\phi(1))=\mathrm{tr}(\psi(1))<1/n$, one has that $$\mathrm{tr}(\xi(1)) > (n-1)/n=n_0/n,$$ and so
$$\phi(1)=\psi(1)\preceq e,$$
where $e=\xi'(1)$. Then the maps $\phi\oplus\xi$ and $\psi\oplus\xi$ have the forms
$$\phi\oplus (n\cdot\xi'),\  \psi\oplus (n\cdot\xi'): A\to \mathrm{M}_{n_0+1}(eQe),$$
respectively.
Note that $eQe$ is stably isomorphic to $Q$, and therefore 
$$\mathrm{K}_0(eQe, \Int/k\Int)=\{0\},\quad k\in\mathbb{N}\setminus\{0\},\quad
\textrm{and}\quad
\mathrm{K}_1(eQe, \Int/k\Int)=\{0\},\quad k\in\mathbb{N}\cup\{0\}.$$
Together with the assumption $[\phi(p)]_0=[\psi(p)]_0$ in $\Kzero(Q)$ for all $p\in\mathcal P$, this implies
$$\phi_\#(p)=\psi_\#(p) \in \underline{\mathrm{K}}(eQe),\quad p\in\widetilde{\mathcal P}.$$
Thus, it follows from Theorem \ref{stable-uniq} that there is a unitary $w\in \mathrm{M}_{n_0+1}(eQe)$ such that
$$
\norm{w^*\left(\begin{array}{cc} \phi(a) & \\ & n_0 \cdot \xi'(a) \end{array} \right)w-\left( \begin{array}{cc} \psi(a) & \\ & n_0 \cdot \xi'(a) \end{array} \right)} < \eps,\quad a\in\mathcal F,
$$
and
\begin{equation}\label{fix-unit} 
w^*(\phi(1)\oplus n_0\cdot e)w=\psi(1)\oplus n_0\cdot e.
\end{equation}
Note that
$$\phi(1)\oplus (n_0\cdot e)=\psi(1)\oplus (n_0\cdot e)=1_Q.$$
By \eqref{fix-unit}, a straightforward calculation shows that 
$$u:=1_Qw1_Q$$
is a unitary of $Q$. Clearly, if the approximation of $\xi$ by $\xi'\oplus\cdots\oplus\xi'$ on $\mathcal F$ is sufficiently good, then, in $Q$,
$$
\norm{u^* (\phi(a)\oplus\xi(a))u-\psi(a)\oplus\xi(a)}<\eps,\quad a\in\mathcal F,
$$
as desired.
\end{proof}

\begin{defn}
Recall that an abelian group $G$ is said to be of (torsion free) rank one if $G\otimes \Ratn \cong \Ratn$. 
\end{defn}

\begin{lem}\label{FDF}
Let $\Delta$ be a compact metrizable Choquet simplex. Then, for any finite subset $\mathcal F\subseteq\aff(\Delta)$ and any $\eps>0$, there exist $m\in\mathbb N$ and unital (pointwise) positive linear maps $\varrho$ and $\theta$,
\begin{displaymath}
\xymatrix{
\aff(\Delta) \ar[r]^-{\varrho} & \Real^m \ar[r]^-{\theta} & \aff(\Delta),
}
\end{displaymath} 
where the unit of $\Real^m$ is $(1, ..., 1)$, 
such that
$$\norm{\theta(\varrho(f))-f}_\infty<\eps,\quad f\in\mathcal F.$$
\end{lem}

\begin{proof}
By Theorem 5.2 of \cite{LL-FDA} and its corollary, there is an increasing sequence of finite-dimensional subspaces of $\aff(\Delta)$ with dense union, containing the canonical order unit $1\in\aff(\Delta)$, and such that each map $\Real^{m_k}\to\Real^{m_{k+1}}$ and $\Real^{m_k}\hookrightarrow\aff(\Delta)$ is positive, with respect to the canonical (pointwise) order relations:
\begin{displaymath}
\xymatrix{
\Real^{m_1} \ar@{^{(}->}[r] & \Real^{m_2} \ar@{^{(}->}[r] & \cdots \ar@{^{(}->}[r] & \aff(\Delta).
}
\end{displaymath}
(The authors are indebted to David Handelman for reminding us of \cite{LL-FDA}.)

Without loss of generality, one may assume that $\mathcal F\subseteq \Real^{m_1}$, and hence one only has to extend the identity map of $\Real^{m_1}$ to a positive unital map $\varrho: \aff(\Delta)\to \Real^{m_1}$.

Write $\Real^{m_1}=\Real e_1\oplus\Real e_2 \oplus \cdots \oplus \Real e_{m_1}$, and consider the unital positive  functionals
$$\rho_i: \Real^{m_1}\ni (x_1, x_2, ..., x_{m_1}) \mapsto x_i\in \Real,\quad i=1, ..., m_1.$$
By the Riesz Extension Theorem (\cite{Riesz-Ext}), each $\rho_i$ can be extended to a unital positive linear functional $\tilde{\rho}_i: \aff(\Delta)\to \Real$. Then the map
$$\varrho: \aff(\Delta) \ni f \mapsto (\tilde{\rho}_1(f), \tilde{\rho}_2(f), ..., \tilde{\rho}_{m_1}(f))\in \Real^{m_1}$$ has the desired property.
\end{proof}

\begin{lem}\label{app-lift}
Let $C=\varinjlim(C_n, \iota_n)$ be a unital inductive system of C*-algebras such that $C$ is simple. Let $(\Real^m,  \norm{\cdot}_\infty, u)$ be  a finite-dimensional ordered Banach space with order unit $u$, and let $\gamma: \Real^m\to\aff(\tr(C))$ be a unital positive linear map. Then, for any finite set $\mathcal F\subseteq\Real^m$ and any $\eps>0$, there are $n$ and a unital positive linear map $\gamma_{n}: \Real^m\to\aff(\tr(C_{n}))$ such that
$$\norm{\gamma(a)-\iota_{n, \infty}\circ\gamma_{n}(a)}<\eps,\quad a\in\mathcal F.$$
\end{lem}
\begin{proof}
Denote by $e_i$, $i=1, ..., m$, the standard basis of $\Real^m$, and write $$u=c_1e_1+\cdots+c_me_m,$$ where $c_1, ..., c_m>0$. Since $C$ is simple, each affine function $\gamma(e_i)$ is strictly positive on $\tr(C)$. Since $\tr(C)$ is compact, there is $\delta_i$ such that 
\begin{equation}\label{lbd-tr}
\gamma(e_i)(\tau)>\delta_i,\quad \tau\in \tr(C), \ 1\leq i\leq m.
\end{equation}

Without loss of generality, one may assume that $\mathcal F=\{e_1, e_2, ..., e_m\}$.

Pick $C_{n}$ and $e'_1, e'_2, ..., e'_{m-1}\in \aff(\tr(C_{n}))$ such that 
$$\norm{\iota_{n, \infty}(e_i')-\gamma(e_i)}_\infty<\min\{\eps, \frac{\delta_i}{2}, \frac{c_m\delta_m}{2(c_1+\cdots+c_{m-1})}, \frac{c_m\eps}{2(c_1+\cdots+c_{m-1})}\},\quad 1\leq i\leq m-1.$$
In particular, by \eqref{lbd-tr},
$$\iota_{n, \infty}(e_i')(\tau)\geq\delta_i/2,\quad \tau\in\tr(C),\  1\leq i\leq m-1.$$
Setting
$$e_m':= \frac{1}{c_m}(1-c_1e'_1-\cdots-c_{m-1}e'_{m-1})\in \aff(\tr(C_{n})),$$
one has
\begin{eqnarray*}
\norm{\iota_{n, \infty}(e_m') - \gamma(e_m) }_\infty & = & \| \frac{1}{c_m}(1-c_1\iota_{n, \infty}(e'_1)-\cdots-c_{m-1}\iota_{n, \infty}(e'_{m-1})) -  \\
&  & \frac{1}{c_m}(1-c_1\gamma(e_1)-\cdots-c_{m-1}\gamma(e_{m-1})) \|_\infty \\
& \leq & \min\{\delta_m/2, \eps\}.
\end{eqnarray*}
In particular, by \eqref{lbd-tr}, $$\iota_{n, \infty}(e'_m)(\tau) \geq \delta_m/2,\quad  \tau\in\tr(C).$$
Then, considering instead the images of $e'_1, e'_2, ..., e'_m$ in a building block further out (replacing $n$ by the later index), one may assume that
$$e'_i(\tau)>\delta_i/4,\quad \tau\in\tr(C_{n}),\ 1\leq i\leq m.$$
In particular, all the affine functions $e_i'\in \aff(\tr(C_{n}))$ are positive.
Define $\gamma_{n}: \Real^m\to \aff(\tr(C_{n_0}))$ by
$$\gamma_{n}(e_i)=e_i',\quad 1\leq i\leq m.$$
It is clear that $\gamma_{n}$ satisfies the condition of the lemma.
\end{proof}

%
%

\begin{thm}\label{tr-factorization}
Let $A$ be a separable simple unital exact C*-algebra satisfying the UCT. Assume that $\tr(A)=\tr_{\mathrm{qd}}(A)$ and that $\Kzero(A)$ is of rank one. Then, for any finite set $\mathcal F\subseteq A\otimes Q$ and any $\eps>0$, there are unital completely positive linear maps $\phi: A\otimes Q \to I$ and $\psi: I\to A\otimes Q$, where $I$ is an interval algebra, such that 
\begin{enumerate}
\item $\phi$ is $\mathcal F$-$\delta$-multiplicative, $\psi$ is an embedding, and
\item $\abs{\tau(\psi\circ\phi(a)-a)}<\eps$, $a\in \mathcal F$, $\tau\in \tr(A\otimes Q)$.
\end{enumerate}
\end{thm}

\begin{proof}
If $\tr(A)=\O$, then the conclusion holds trivially (with $I=\{0\}$). Otherwise, assuming, as we may, that $A\cong A\otimes Q$, we have $\Kzero(A)\cong\Ratn$ (as order-unit groups).

Apply Corollary \ref{stable-uniq-Q} to $A$ with respect to $({\mathcal F}\cdot{\mathcal F}, \eps/4)$ to obtain $n$ and $(\mathcal P, \mathcal G, \delta)$. Since $\Kzero(A)=\Ratn$ (unique unital identification), we may suppose that $\mathcal P=\{1_A\}$.

By Theorem 3.9 of \cite{Thomsen-AI-Tr}, there is a simple unital inductive limit $C=\varinjlim(C_i, \iota_i)$ such that $\Kzero(C)=\Ratn$ (unital identification), $C_i=\MC{k_i}{[0, 1]}$, the maps $\iota_i$ are injective, and there is an isomorphism
$$\Xi: \aff(\tr(A))\cong\aff(\tr(C)).$$ 

By Lemma \ref{FDF}, there is an approximate factorization, by means of unital positive maps, 
\begin{displaymath}
\xymatrix{
\aff(\tr(A)) \ar[r]^-{\varrho} & \Real^m \ar[r]^-{\theta} & \aff(\tr(A)),
}
\end{displaymath} 
such that
$$\|{\theta(\varrho(\hat{f}))-\hat{f}}\|_\infty<\eps/16,\quad f\in\mathcal F.$$
Therefore, by Lemma \ref{app-lift}, after discarding finitely many terms of the sequence $(C_i, \iota_i)$, there is a unital positive linear map 
\begin{displaymath}
\xymatrix{
\gamma: \aff(\tr(A))\ar[r]^-\varrho & \Real^m \ar[r] & \aff(\tr(C_1))
}
\end{displaymath}
such that 
\begin{equation}\label{tr-lift}
\|{(\iota_{1, \infty})_*(\gamma(\hat{f}))-\Xi(\hat{f})}\|_\infty<\eps/8,\quad f\in\mathcal F.
\end{equation}

Denote by $\gamma^*: \tr(C_1)\to\tr(A)$ the affine map induced by $\gamma$ on tracial simplices. Since $\gamma$ factors through $\Real^m$ (so that $\gamma^*$ factors through a finite dimensional simplex), there are $\tau_1, ..., \tau_m \in \tr(A)$ and continuous functions $c_1, c_2, ..., c_m: [0, 1] \to [0, 1]$ such that
\begin{equation}\label{tr-span}
\gamma^*(\tau_t)=c_1(t)\tau_1+c_2(t)\tau_2+\cdots+c_m(t)\tau_m,\quad t\in[0, 1],
\end{equation}
and
$$c_1(t)+c_2(t)+\cdots+c_m(t)=1,\quad t\in[0, 1],$$
where $\tau_t\in\tr(C_1)$ is determined by the Dirac measure concentrated at $t\in[0, 1]$.

Since $\tau_1, \tau_2, ..., \tau_m\in\tr_{\mathrm{qd}}(A)$, there are unital completely positive linear maps $\phi_k: A\to Q$, $k=1, 2, ..., m$, such that 
each $\phi_k$ is $\mathcal G$-$\delta$-multiplicative, and
\begin{equation}\label{QD-tr}
\abs{\mathrm{tr}(\phi_k(f))-\tau_k(f)}<\eps/16m,\quad f\in \mathcal F. 
\end{equation}

For each $t\in[0, 1]$, there is a open neighbourhood $U$ such that for any $s\in U$, one has
$$\abs{c_k(s)-c_k(t)}<1/4mn.$$
(Recall that $n$ is the constant from Corollary \ref{stable-uniq-Q}, as in the second paragraph of the proof.)
Since $[0, 1]$ is compact, there is a partition
$0=t_0<t_1<\cdots<t_{l-1} < t_l=1$
such that 
\begin{equation}\label{dense-cut-1}
\abs{c_k(s)-c_k(t_j)}< 1/4mn,\quad s\in [t_{j-1}, t_{j}].
\end{equation}
Moreover, we may assume that this partition is fine enough that
\begin{equation}\label{sm-vr-tr}
|{\gamma(\hat{f})(\tau_t)-\gamma(\hat{f})(\tau_{t_j})}|<\eps/8,\quad f\in\mathcal F,\ t\in[t_{j-1}, t_j].
\end{equation}

For each $j=0, 1, ..., l$, 
%
pick rational numbers $r_{j,1}, r_{j, 2}, ..., r_{j,m}\in[0, 1]$ such that 
$$r_{j, 1}+\cdots+r_{j, m}=1$$ 
and
\begin{equation}\label{pert-coef-ratn}
\abs{r_{j, k}-c_k(t_j)}<\min\{\eps/16m, 1/4mn\},\quad k=1, ..., m.
\end{equation}
Write $r_{j, k}=q_{j, k}/p$ where $q_{j, k}, p\in\mathbb N$, and then define
$$\varphi_j:=(\underbrace{\phi_1\oplus\cdots\oplus\phi_1}_{q_{j, 1}})\oplus\cdots\oplus (\underbrace{\phi_m\oplus\cdots\oplus\phi_m}_{q_{j, m}}): A\to Q.$$
Note that it follows from \eqref{tr-span}, \eqref{QD-tr}, and \eqref{pert-coef-ratn} that
\begin{equation}\label{lift-pt}
\abs{\mathrm{tr}(\varphi_j(f))-\gamma^*(\tau_{t_j})(f)}< \eps/4,\quad  f\in\mathcal F.
\end{equation}


By \eqref{pert-coef-ratn}, \eqref{dense-cut-1}, one has that
\begin{equation}\label{ratn-coef-close}
\frac{\abs{q_{j, k} - q_{j+1, k}}}{p}< \frac{1}{mn},\quad k=1, ..., m,\ j=0, ..., l-1.
\end{equation}

For each $j=0, ..., l-1$, compare the direct sum maps
$$\varphi_j=(\underbrace{\phi_1\oplus\cdots\oplus\phi_1}_{q_{j, 1}})\oplus\cdots\oplus (\underbrace{\phi_m\oplus\cdots\oplus\phi_m}_{q_{j, m}})  $$
and
$$\varphi_{j+1}=(\underbrace{\phi_1\oplus\cdots\oplus\phi_1}_{q_{j+1, 1}})\oplus\cdots\oplus (\underbrace{\phi_m\oplus\cdots\oplus\phi_m}_{q_{j+1, m}}),$$
and consider the common direct summand of these two maps,
$$\psi_{j}:=(\underbrace{\phi_1\oplus\cdots\oplus\phi_1}_{\min\{q_{j, 1}, q_{j+1, 1}\}})\oplus\cdots\oplus (\underbrace{\phi_m\oplus\cdots\oplus\phi_m}_{\min\{q_{j, m}, q_{j+1, m}\}}).$$

By \eqref{ratn-coef-close}, one has
$$\abs{\mathrm{tr}(1-\psi_j(1))} = \frac{1}{p}\sum_{k=1}^m\abs{q_{j, k} - q_{j+1, k}}<\frac{1}{n}.$$
On the other hand, since $\varphi_j$ and $\varphi_{j+1}$ are unital, 
one has 
$$[(\varphi_j\ominus\psi_j)(1_A)]_0=1-\mathrm{tr}(\psi_j(1_A))=[(\varphi_{j+1}\ominus\psi_j)(1_A)]_0.$$
Recall that $\mathcal P=\{1_A\}$.
By the conclusion of Corollary \ref{stable-uniq-Q} 
there is a unitary $u_{j+1}$ such that
$$\norm{\varphi_j(f)-u_{j+1}^*\varphi_{j+1}(f)u_{j+1}}<\eps/4,\quad f\in\mathcal F\cdot\mathcal F,\  0\leq j\leq l-1.$$
Define $v_0=1$, and set
$$u_ju_{j-1} \cdots u_1=v_j,\quad j=1, ..., l.$$
Then, for any $0\leq j\leq l-1$ and any $f\in\mathcal F\cdot\mathcal F$, one has
\begin{eqnarray*}
&&\norm{\mathrm{Ad}(v_j)\circ\varphi_i(f)-\mathrm{Ad}(v_{j+1})\circ\varphi_{j+1}(f)}\\
&=& \norm{(u_j \cdots u_1)^*\varphi_i(f)(u_j \cdots u_1)-(u_{j+1} \cdots u_1)^*\varphi_{j+1}(f)(u_{j+1} \cdots u_1)}\\
&=& \norm{\varphi_j(f)-u_{j+1}^*\varphi_{j+1}(f)u_{j+1}}<\eps/4.
\end{eqnarray*}
Replacing each homomorphism $\varphi_j$ by $\mathrm{Ad}(v_j)\circ\varphi_j$ for $j=1, ..., l$, and still denoting it by $\varphi_j$, one has
\begin{equation}\label{close-nb}
\norm{\varphi_j(f)-\varphi_{j+1}(f)}<\eps/4,\quad f\in\mathcal F\cdot\mathcal F,\ 0\leq j\leq l-1.
\end{equation}
Define a unital completely positive linear map $\phi: A\to C_1$ by
$$\phi(f)(t):=\frac{t_{j+1}-t}{t_{j+1}-t_j}\varphi_j(f) + \frac{t-t_j}{t_{j+1}-t_j}\varphi_{j+1}(f),\quad \textrm{if $t\in[t_j, t_{j+1}]$}.$$
Then, by \eqref{close-nb}, the map $\phi$ is $\mathcal F$-$\eps$-multiplicative. By \eqref{lift-pt} and \eqref{sm-vr-tr}, one has
\begin{equation}\label{tr-est-AC}
\|{\phi_*(\hat{f}) - \gamma(\hat{f})}\|_\infty < \eps/2,\quad f\in\mathcal F.
\end{equation}

Note that $A$ and $C$ have cancellation for projections, and also $\Kzero^+(A)=\Kzero^+(C)=\Ratn^+$ (unital identification) and $\aff(\tr(A))\cong\aff(\tr(C))$. By Theorem 4.4 and Corollary 6.8 of \cite{ESR-Cuntz} (see also Theorem 2.6 of \cite{Brn-Toms-3app} and Theorem 5.5 of \cite{BPT-Cuntz}, expressed in terms of W instead of Cu), it follows that the Cuntz semigroup of $A$ and the Cuntz semigroup of $C$ are isomorphic. Applied to the canonical unital map $\mathrm{Cu}(C_1)\to\mathrm{Cu}(C)\cong\mathrm{Cu}(A)$, Theorem 1 of \cite{Robert-Cu} implies that there is a unital homomorphism $\psi: C_1\to A$ giving rise to this map, and in particular such that 
\begin{equation}\label{tr-est-CA}
\psi_*=\Xi^{-1}\circ(\iota_{1, \infty})_*\quad\textrm{on $\aff(\tr(C_1))$}.
\end{equation}
Since the ideal of $\mathrm{Cu}(C_1)$ killed by the map $\mathrm{Cu}(C_1)\to\mathrm{Cu}(C)\cong\mathrm{Cu}(A)$ is zero, as the map $C_1\to C$ is an embedding, it follows that the map $C_1\to A$ is also an embedding. 
By \eqref{tr-est-AC}, \eqref{tr-est-CA}, and \eqref{tr-lift}, one then has 
$$\|\phi_*\circ\psi_*(\hat{f})-\hat{f}\|_\infty<\eps,\quad f\in\mathcal F,$$
as desired.
\end{proof}

Recall that 
\begin{defn}[\cite{LinTAF1}, \cite{EN-TApprox}]\label{defn-TAI}
Let $\mathcal S$ be a class of unital C*-algebras. A C*-algebra $A$ is said to be tracially approximated by the C*-algebras in $\mathcal S$, and one writes $A\in \mathrm{TA}\mathcal S$, if the following condition holds: For any finite set $\mathcal F\subseteq A$, any $\eps>0$, and any non-zero $a\in A^+$, there is a non-zero sub-C*-algebra $S\subseteq A$ such that $S\in\mathcal S$,  and if $p=1_S$, then
\begin{enumerate}
\item $\norm{pf-fp}<\eps$, $f\in\mathcal F$,
\item $pfp\in_\eps S$, $f\in\mathcal F$, and
\item $1-p$ is Murray-von Neumann equivalent to a subprojection of $\overline{aAa}$.
\end{enumerate}

Denote by $\mathcal I$ the class of interval algebras, i.e., $$\mathcal I=\{\mathrm{C}([0, 1])\otimes F: \textrm{$F$ is a finite dimensional C*-algebra}\}.$$ $\mathrm{TA}\mathcal I$, then, is the class of C*-algebras which can be tracially approximated by interval algebras.
\end{defn}

For $\mathrm{TA}\mathcal I$ algebras, based on Winter's deformation technique (\cite{Winter-Z} and \cite{Lin-App}) and on \cite{L-N},  one has the following classification theorem.
\begin{thm}[Corollary 11.9 of \cite{Lnclasn}]\label{clasnTAI}
Let $A, B$ be unital separable amenable simple C*-algebras satisfying the UCT. Assume that $A$, $B$ are Jiang-Su stable, and assume that $A\otimes Q\in\mathrm{TA}\mathcal I$ and $B\otimes Q\in \mathrm{TA}\mathcal I$. Then $A\cong B$ if and only if $\mathrm{Ell}(A)\cong\mathrm{Ell}(B)$.
\end{thm}

The following is the main result of this note, which asserts that certain abstract C*-algebras are covered by the classification theorem above.

\begin{thm}\label{main-cor}
Let $A$ be a separable simple unital C*-algebra satisfying the UCT. Assume that $A\otimes Q$ has finite nuclear dimension, $\tr(A)=\tr_{\mathrm{qd}}(A)$, and $\Kzero(A)\otimes\Ratn = \Ratn$ (identification of order-unit groups). Then $A\otimes Q\in\mathrm{TA}\mathcal I$. 
\end{thm}
\begin{proof}
This follows from Theorem \ref{tr-factorization} above and Theorem 2.2 of \cite{Winter-TA} directly.
\end{proof}

\begin{proof}[Proof of Theorem \ref{main-thm}]
By Proposition 8.5 of \cite{BBSTWW},  as $A\otimes Q$ has finite decomposition rank, $\tr(A\otimes Q)=\tr_{\mathrm{qd}}(A\otimes Q)$. Furthermore, by \cite{KW-DR-QD}, $A\otimes Q$ is stably finite and nuclear and so by \cite{BlaTrace} and \cite{Haagtrace}, $\tr(A)\neq\O$. Then $\Kzero(A\otimes Q)=\Ratn$ (as order-unit groups), and the statement follows from Theorem \ref{main-cor}. (The classifiability of $A\otimes\mathcal Z$ holds by Theorem \ref{clasnTAI}.)
\end{proof}

\bibliographystyle{plainurl}

\end{document}